\documentclass[12pt]{amsart}
\usepackage{amsxtra}
\usepackage{graphics}
\usepackage{latexsym}
\usepackage{enumerate}
\usepackage{xcolor}
\usepackage{amsmath}
\usepackage{amssymb,amsthm,amsfonts}
\usepackage{mathtools}
\usepackage{amscd}
\usepackage[arrow, matrix, curve]{xy}
\usepackage[mathscr]{euscript}
\usepackage{syntonly}
\usepackage{ifpdf}
\usepackage{hyperref}
\title[A note on the filtered decomposition theorem]{A note on the filered decomposition theorem}

\author{Zebao Zhang
}
\address{
University of Science and Technology of China, School of Mathematical Sciences, Hefei, 230026, China}

\address{Peking University,
Beijing International Center for Mathematical Research, Beijing, 100871, China
}
\email{zzb2@mail.ustc.edu.cn}

\begin{document}
\theoremstyle{plain}
\newtheorem{thm}{Theorem}[section]
\newtheorem{theorem}[thm]{Theorem}
\newtheorem*{theorem*}{Theorem}
\newtheorem*{theoremA*}{Theorem A}
\newtheorem*{theoremB*}{Theorem B}
\newtheorem*{theoremC*}{Theorem C}
\newtheorem*{definition*}{Definition}
\newtheorem{lemma}[thm]{Lemma}
\newtheorem{sublemma}[thm]{Sublemma}
\newtheorem{corollary}[thm]{Corollary}
\newtheorem*{corollary*}{Corollary}
\newtheorem{proposition}[thm]{Proposition}
\newtheorem{addendum}[thm]{Addendum}
\newtheorem{variant}[thm]{Variant}
\theoremstyle{definition}
\newtheorem{lemma-definition}[thm]{Lemma-Definition}
\newtheorem{proposition-definition}[thm]{Proposition-Definition}
\newtheorem{theorem-definition}[thm]{Theorem-Definition}
\newtheorem{construction}[thm]{Construction}
\newtheorem{notations}[thm]{Notations}
\newtheorem{question}[thm]{Question}
\newtheorem{problem}[thm]{Problem}
\newtheorem*{problem*}{Problem}
\newtheorem{remark}[thm]{Remark}
\newtheorem*{remark*}{Remark}
\newtheorem{remarks}[thm]{Remarks}
\newtheorem{definition}[thm]{Definition}
\newtheorem{claim}[thm]{Claim}
\newtheorem{assumption}[thm]{Assumption}
\newtheorem{assumptions}[thm]{Assumptions}
\newtheorem{properties}[thm]{Properties}
\newtheorem{example}[thm]{Example}
\newtheorem{conjecture}[thm]{Conjecture}


\newcommand{\sX}{{\mathcal X}}
\newcommand{\sY}{{\mathcal Y}}
\newcommand{\sZ}{{\mathcal Z}}
\newcommand{\A}{{\mathbb A}}
\newcommand{\B}{{\mathbb B}}
\newcommand{\C}{{\mathbb C}}
\newcommand{\D}{{\mathbb D}}
\newcommand{\E}{{\mathbb E}}
\newcommand{\F}{{\mathbb F}}
\newcommand{\G}{{\mathbb G}}
\newcommand{\HH}{{\mathbb H}}
\newcommand{\I}{{\mathbb I}}
\newcommand{\J}{{\mathbb J}}
\renewcommand{\L}{{\mathbb L}}
\newcommand{\M}{{\mathbb M}}
\newcommand{\N}{{\mathbb N}}
\renewcommand{\P}{{\mathbb P}}
\newcommand{\Q}{{\mathbb Q}}
\newcommand{\R}{{\mathbb R}}
\newcommand{\SSS}{{\mathbb S}}
\newcommand{\T}{{\mathbb T}}
\newcommand{\U}{{\mathbb U}}
\newcommand{\V}{{\mathbb V}}
\newcommand{\W}{{\mathbb W}}
\newcommand{\X}{{\mathbb X}}
\newcommand{\Y}{{\mathbb Y}}
\newcommand{\Z}{{\mathbb Z}}
\newcommand{\id}{{\rm id}}
\newcommand{\rank}{{\rm rank}}
\newcommand{\END}{{\mathbb E}{\rm nd}}
\newcommand{\End}{{\rm End}}
\newcommand{\Hom}{{\rm Hom}}
\newcommand{\Hg}{{\rm Hg}}
\newcommand{\tr}{{\rm tr}}
\newcommand{\Sl}{{\rm Sl}}
\newcommand{\Gl}{{\rm Gl}}
\newcommand{\Cor}{{\rm Cor}}
\newcommand{\Aut}{\mathrm{Aut}}
\newcommand{\dR}{\mathrm{dR}}
\newcommand{\Hig}{\mathrm{Hig}}
\newcommand{\Sym}{\mathrm{Sym}}
\newcommand{\ModuliCY}{\mathfrak{M}_{CY}}
\newcommand{\HyperCY}{\mathfrak{H}_{CY}}
\newcommand{\ModuliAR}{\mathfrak{M}_{AR}}
\newcommand{\Modulione}{\mathfrak{M}_{1,n+3}}
\newcommand{\Modulin}{\mathfrak{M}_{n,n+3}}
\newcommand{\Gal}{\mathrm{Gal}}
\newcommand{\Spec}{\mathrm{Spec}}
\newcommand{\res}{\mathrm{res}}
\newcommand{\coker}{\mathrm{coker}}
\newcommand{\Jac}{\mathrm{Jac}}
\newcommand{\HIG}{\mathrm{HIG}}
\newcommand{\MIC}{\mathrm{MIC}}
\newcommand{\Gr}{\mathrm{Gr}}
\newcommand{\Fil}{\mathrm{Fil}}
\maketitle
\noindent{\small {\bf Abstract.}
We generalize the logarithmic decomposition theorem of Deligne-Illusie to a filtered version. There are two applications. The easier one provides a mod $p$ proof for a vanishing theorem in characteristic zero. The deeper one  gives rise to  a positive characteristic analogue of a theorem of Deligne on the mixed Hodge structure attached to complex algebraic varieties.


\setcounter{tocdepth}{1}
\tableofcontents

\section{Introduction}
Let $k$ be a perfect field of characteristic $p>0$,  and let $X$ be a smooth $k$-scheme equipped with a normal crossing divisor (NCD) $D\subset X$. Assume that $(X,D)$ is $W_2(k)$-liftable. By $\tau_{<p},F$, we mean the canonical truncation of complexes at degrees $<p$ and the relative Frobenius of $X$ over $k$, respectively. Let $(X',D')/k$ be the base change of $(X,D)$ via the Frobenius automorphism $F_k$ of $\mathrm{Spec}(k)$. Denote by $D^+(X')$ the bounded below derived category of $\mathcal{O}_{X'}$-modules. Deligne-Illusie \cite[Theorem 2.1]{DI} (see also \cite[Theorem 5.1]{IL02}, \cite[Theorem 10.16]{EV}) showed the following cohomological Hodge decomposition in positive characteristic.
\begin{theorem}[Deligne-Illusie]\label{D-I, decomposition}
$\tau_{<p}F_{*}\Omega^*_{X/k}(\log D)$ is decomposable in $D^+(X')$, i.e., we have
\begin{eqnarray}\label{DI DT}
\tau_{<p}F_{*}\Omega^*_{X/k}(\log D)\cong\bigoplus_{i=0}^{p-1}\Omega^i_{X'/k}(\log D')[-i]~\mathrm{in}~D^+(X').
\end{eqnarray}
\end{theorem}
There are various extensions of Theorem \ref{D-I, decomposition}, such as Ogus-Vologodsky \cite[Corollary 2.27]{OV}, Schepler \cite[Theorem 5.7]{Schepler}.  Recently, Illusie kindly informed us that he has generalized Theorem \ref{D-I, decomposition} to a lci scheme over $k$ by the derived de Rham complex. In this paper, we extend Theorem \ref{D-I, decomposition} to a filtered version,  following the spirit of Deligne-Illusie \cite{DI}. 
As shown in classical Hodge theory, one can define a weight filtration $W$ on the logarithmic de Rham complex $\Omega^*_{X/k}(\log D)$ in positive characteristic. On the other hand, one can use the similar way to construct a weight filtration  on $\bigoplus_{i}\Omega^i_{X'/k}(\log D')[-i]$,  again denoted by $W$. 
Note that both sides of \eqref{DI DT} inherit a filtration,  by abuse of notation, we still write them by the same letter $W$. Denote by $D^+F(X')$ the bounded below derived category of filtered $\mathcal{O}_{X'}$-modules. We prove a filtered version of Theorem \ref{D-I, decomposition}, which can be regarded as a cohomological mixed Hodge decomposition in positive characteristic.
\begin{theorem}\label{FSDT}
$(\tau_{<p}F_{*}\Omega^*_{X/k}(\log D),W)$ is decomposable in $D^+F(X')$. More precisely, for any $W_2(k)$-lifting $(\tilde X,\tilde D)/W_2(k)$ of $(X,D)/k$, we have an isomorphism in $D^+F(X')$:
 \begin{eqnarray*}
\Psi_{(\tilde X,\tilde D)}:(\bigoplus_{i=0}^{p-1}\Omega^i_{X'/k}(\log D')[-i],W)\to(\tau_{<p}F_{*}\Omega^*_{X/k}(\log D),W).
\end{eqnarray*}
\end{theorem}
There is an easy application of Theorem \ref{FSDT}, which gives rise to a mod $p$ proof of the following theorem.
\begin{theorem}\label{vanishing theorem}
Let $X$ be a smooth projective variety over a field $K$ of characteristic $0$, $D\subset X$ a NCD, $L$ an ample line bundle on $X$. Then for any $l$, we have
\begin{eqnarray*}
\mathbb{H}^i(X, W_l\Omega_{X/K}^j(\log D)\otimes L)=0\quad\mathrm{for}~i+j>\mathrm{dim}(X).
\end{eqnarray*}
\end{theorem}
Note that if $D$ is a simple normal crossing divisor (SNCD), then this theorem is an easy consequence of Kodaira-Akizuki-Nakano vanishing theorem (see \cite[Theorem 6.10]{IL02}). 

Another deeper application of Theorem \ref{FSDT} concerns a de Rham analogue of the following theorem in positive characteristic.
\begin{theorem}[Deligne {{\cite[Theorem  8.2.2]{D3}}}]\label{Deligne}The cohomology of smooth complex algebraic varieties is endowed with a mixed Hodge structure (MHS) functorial with respect to algebraic maps.
\end{theorem}
We briefly review Deligne's proof. Let $X$ be temporarily a smooth complex algebraic variety. By Nagata's compactification theorem and Hironaka's desingularization theorem, there exists a smooth proper complex algebraic variety $\overline X$ endowed with a simple normal crossing divisor $D$ such that $X=\overline X-D$. It can be shown that the cohomology of $X$ is computed by the hypercohomology of the logarithmic de Rham complex $\Omega^*_{\overline X/\mathbb{C}}(\log D)$. This complex carries two filtrations, i.e., the weight filtration $W$ and the Hodge filtration $\Fil$. A key result says that the bifiltered complex $(\Omega^*_{\overline X/\mathbb{C}}(\log D),W,\Fil)$ is a cohomological mixed Hodge complex. It follows that
$$(K^*,W,\Fil):=R\Gamma(X,(\Omega^*_{\overline X/\mathbb{C}}(\log D),W,\Fil))$$
 is a Hodge complex. Denote by $\{E^{i,j}_r,d_r^{i,j}\}_{r\geq0}$ the spectral sequence of $(K^*,W)$. Then:
 \begin{itemize}
 \item[-\ ] the filtration $\Fil$ on $K^*$ induces a recursive filtration $\Fil_{\mathrm{rec}}$ on $E^{i,j}_r$ such that for any $r\geq 1$, $(E^{i,j}_r,\Fil_{\mathrm{rec}})$ is a pure Hodge structure and $d_r^{i,j}$ is a morphism of Hodge structures;
 \item[-\ ] this spectral sequence degenerates at $E_2$;
 \item[-\ ] $\{E^{i,j}_{\infty}\}$ gives rise to the mixed Hodge structure of the cohomology of $X$.
 \end{itemize}
The discussion above is functorial with respect to algebraic maps.

In positive characteristic $p$, we try to follow the line of Deligne's proof. Let $\mathrm{Sch}^{\mathrm{sp}}_{W_2(k),\log}$ be the category of smooth proper $W_2(k)$-schemes of dimension less than $p$ endowed with a normal crossing divisor. Pick $(\tilde X,\tilde D)\in \mathrm{Sch}^{\mathrm{sp}}_{W_2(k),\log}$ with closed fiber $(X,D)/k$. We introduce the notion of cohomological mixed Fontaine-Laffaille complexes over $X$, which can be regarded as a positive characteristic analogue of the notion of cohomological mixed Hodge complexes. In particular, we show that $(\Omega^*_{X/k}(\log D),W,\Fil,\Psi_{(\tilde X,\tilde D)})$ is such a  bifiltered complex. Set 
$$(K^*_\dR,W,\Fil,\psi):=R\Gamma(X,(\Omega^*_{X/k}(\log D),W,\Fil,\Psi_{(\tilde X,\tilde D)})),$$
then we obtain a mixed Fontaine-Laffaille complex over $k$. Denote by
 \begin{eqnarray}\label{E functor}
 \{E_{r,\dR}^{-i,j},d^{-i,j}_{r,\dR}\}_{r\geq1},~ \{E_{r,\Hig}^{-i,j},d^{-i,j}_{r,\Hig}\}_{r\geq 1}
 \end{eqnarray}
 the spectral sequence of $(K_\dR^*,W),(K_\Hig^*,W):=F_k^*\Gr_\Fil(K_\dR,W)$, respectively. Then:
 \begin{itemize}
 \item[-\ ] the filtration $\Fil$ on $K_\dR^*$ induces a recursive filtration $\Fil_{\mathrm{rec}}$ on $E^{i,j}_{r,\dR}$ such that for any $r\geq 1$, $F_k^*\Gr_{\Fil_{\mathrm{rec}}}$ sends $d^{i,j}_{r,\dR}$ to $d^{i,j}_{r,\Hig}$;
  \item[-\ ] $\psi$ induces an isomorphism of spectral sequences
  $$\psi:\{E_{r,\Hig}^{-i,j},d^{-i,j}_{r,\Hig}\}_{r\geq 1}\to\{E_{r,\dR}^{-i,j},d^{-i,j}_{r,\dR}\}_{r\geq1}.$$
 \end{itemize}
 Unfortunately, we can not prove the $E_2$-degeneration of $ \{E_{r,\dR}^{-i,j},d^{-i,j}_{r,\dR}\}_{r\geq1}$ as in classical Hodge theory, since there has no appropriate positive characteristic  analogue of the action of complex conjugation on Hodge structures. Using Lemma \ref{homotopy formula}, we show that the discussion above is functorial with respect to morphisms in $\mathrm{Sch}^{\mathrm{sp}}_{W_2(k),\log}$.
 
 Set $E_{r\geq 1}/\mathrm{Vec}_k$ to be the addtive category of spectral sequences starting from $E_1$-page over the category $\mathrm{Vec}_k$ of $k$-vector spaces. Note that the Frobenius automorphism $F^*_k$ of $k$ induces an automorphism of $E_{r\geq 1}/\mathrm{Vec}_k$, again denoted by $F_k^*$. The following theorem is a positive characteristic analogue of Theorem \ref{Deligne}.
  \begin{theorem}\label{main theorem}
 \eqref{E functor} defines functors
  \begin{eqnarray*}\label{dR Hig functors}
\rho_\dR,~\rho_{\Hig}:(\mathrm{Sch}^{\mathrm{sp}}_{W_2(k),\log})^{\mathrm{op}}\to E_{r\geq1}/\mathrm{Vec}_k.
  \end{eqnarray*}
  There have a natural Hodge filtration $\Fil$ on $\rho_\dR$ such that
  \begin{eqnarray*}
  \rho_{\Hig}=F_k^*\circ\Gr_\Fil\rho_{\dR}
  \end{eqnarray*}
  and an isomorphism of functors
  $$\psi:\rho_\Hig\to\rho_\dR.$$
We call $(\rho_\dR,\rho_\Hig,\Fil,\psi)$ a Fontaine-Laffaille quadruple over the category  $\mathrm{Sch}^{\mathrm{sp}}_{W_2(k),\log}$.
\end{theorem}

\subsection*{Acknowledgements}
I am deeply indebted to Luc Illusie, Junchao Shentu, Mao Sheng, Jilong Tong, Qizheng Yin, Lei Zhang,  Kang Zuo for their useful comments and suggestions on this paper. I appreciate the highly effective help of the anonymous referees which improves the readability  of this paper in a large extent. I also thank Jijian Song, Ziyan Song and Zeping Zhu. Finally, I want to show my gratitude to Zhiyu Tian, Qizheng Yin and Beijing International Center for Mathematical Research for supporting me as a visitor about two years.

\section{Filtered decomposition theorem}
\subsection{The decomposition theorem of Deligne-Illusie }
 In this subsection, we briefly recall the decomposition theorem of Deligne-Illusie. For more details, we refer the reader to \cite[\S8-10]{EV}, \cite[\S5]{IL02}. Let $k$ be a perfect field of positive characteristic $p$, $X$ a smooth variety over $k$, $D\subset X$ a normal crossing divisor. Assume that there exists a $W_2(k)$-lifting $(\tilde X,\tilde D)$ of $(X,D)$. Consider the following famous diagram:
$$\xymatrix{X\ar[r]^-{F_{X/k}}\ar[1,1]\ar@/^2pc/[rr]^-{F_X}&X'\ar[r]^-{\pi_{X/k}}\ar[d]&X\ar[d]\\
&\mathrm{Spec}(k)\ar[r]^-{F_k}&\mathrm{Spec}(k),}$$
where the right square is cartesian, $F_k$ and $F_X$  the absolute Frobenius of $\mathrm{Spec}(k)$ and $X$, respectively. Set $F:=F_{X/k}$. Note that we may take $X':=X$ (i.e., $\pi_{X/k}=\mathrm{id}_X$) as abstract scheme \footnote{The structure morphism $X'\to\mathrm{Spec}(k)$ is the composite of $\xymatrix{X\ar[r]&\mathrm{Spec}(k)\ar[r]^-{F_k^{-1}}&\mathrm{Spec}(k)}$.}, hence $(\tilde X,\tilde D)$ induces a $W_2(k)$-lifting $(\tilde X',\tilde D')$ of $(X',D':=D)$.
Choose an open affine covering $\mathcal{U}=\{U_\alpha\}$ of $X$, and let $\mathcal{U}'=\{U'_\alpha\},\tilde{\mathcal{U}}=\{\tilde{U}_\alpha\},\tilde{\mathcal{U}}'=\{\tilde U'_\alpha\}$ be the corresponding coverings of $X',\tilde X,\tilde X'$, respectively.  By smoothness and affineness, there exists a family of log liftings $\{\tilde{F}_{U_\alpha/k}:\tilde U_\alpha\to\tilde{U}'_\alpha\}$, i.e., $F_{U_\alpha/k}=\tilde{F}_{U_\alpha/k}\times_{W_2(k)}k$ and $\tilde{F}_{U_\alpha/k}^*\tilde D'=p\tilde D$ for any $\alpha$.  
\begin{definition}
The family $\{U_\alpha,\tilde F_{U_\alpha/k}\}$ constructed above is said to be a D-I atlas of $(\tilde X,\tilde D)/W_2(k)$.
\end{definition}
\begin{lemma}[{{\cite[ \S2]{DI}}}]\label{key lemma} 
Using a D-I atlas $\{U_\alpha,\tilde F_{U_\alpha}\}$ of $(\tilde X,\tilde D)/W_2(k)$, one can construct two families 
$$\begin{array}{c}\{\zeta_{\alpha}:\Omega_{X'/k}(\log D')|_{U'_{\alpha}}\to F_{*}\Omega_{X/k}(\log D)|_{U'_{\alpha}}\},\\
\{h_{\alpha\beta}:\Omega_{X'/k}(\log D)|_{U'_{\alpha}\cap U'_{\beta}}\to F_{*}\mathcal{O}_X|_{U'_{\alpha}\cap U'_{\beta}}\}\end{array}$$ such that:
\begin{itemize}
\item[(i)\ ] $\zeta_{\beta}-\zeta_{\alpha}=dh_{\alpha\beta}$;
\item[(ii)\ ] (cocycle condition) $h_{\alpha\beta}+h_{\beta\gamma}=h_{\alpha\gamma}$.
\end{itemize}
\end{lemma}
\begin{proof}
We briefly recall the constructions of $\zeta_{\alpha}$ and $h_{\alpha\beta}$. Let $x\in\mathcal{O}_{ U_\alpha'}$ with a lifting $\tilde x\in\mathcal{O}_{\tilde U_\alpha'}$. It is easy to see that $\tilde F_{U_\alpha}(\tilde x)=\tilde x^p+p\tilde\lambda$, where $\tilde\lambda\in\mathcal{O}_{\tilde U_\alpha'}$ is a lifting of some $\lambda\in\mathcal{O}_{ U_\alpha'}$. We set 
$$\zeta_\alpha(dx):=x^{p-1}dx+d\lambda.$$
Furthermore, we assume that $x\in\mathcal{O}_{U_\alpha'\cap U_\beta'}$. Write $\tilde F_{U_\beta}(\tilde x)=\tilde x^p+p\tilde\mu$, where $\tilde\mu\in\mathcal{O}_{\tilde U_\alpha'\cap\tilde U_\beta'}$ is a lifting of some $\mu\in\mathcal{O}_{ U_\alpha'\cap  U_\beta'}$. We set
$$h_{\alpha\beta}(dx):=\mu-\lambda.$$
One can check that $\zeta_\alpha,h_{\alpha\beta}$ are well-defined and they satisfy the conditions (i), (ii).
\end{proof}
The trick for showing Theorem \ref{D-I, decomposition} can be summarized by the following diagram of quasi-isomorphisms:
\begin{eqnarray}\label{varphi}
\xymatrix{\tau_{<p} F_{*}\Omega^*_{X/k}(\log D)\ar[r]^-{\iota}&\check{\mathcal{C}}(\mathcal{U}', \tau_{<p}F_{*}\Omega^*_{X/k}(\log D))\\
&\bigoplus_{i=0}^{p-1}\Omega^i_{X'/k}(\log D')[-i].\ar[u]^-{\varphi_{X}}}
\end{eqnarray}
This diagram involves many notations, we quickly review them. 
\begin{itemize}
\item[-\ ] For a complex $K^*$ over some abelian category, we set $\tau_{<p}K^*$ to be the subcomplex of $K^*$ with components $K^i$ for $i\leq p-2$, $\mathrm{ker}(d:K^{p-1}\to K^p)$ for $i=p-1$ and $0$ for $i\geq p$.
\item[-\ ] $\Omega^*_{X/k}(\log D)$ is the logarithmic de Rham complex of $X$ over $k$ with pole along $D$ (see for instance \cite[\S7.1]{IL02}).
\item[-\ ]Let $\check{\mathcal{C}}^*(\mathcal{U}',\tau_{<p}F_{*}\Omega^*_{X/k}(\log D))$ be the \v{C}ech resolution of $\tau_{<p}F_{*}\Omega^*_{X/k}(\log D)$ by $\mathcal{U}'$, and let $\check{\mathcal{C}}(\mathcal{U}',\tau_{<p}F_{*}\Omega^*_{X/k}(\log D))$ be the associated single complex \cite[Pages 121-122]{IL02}.
\item[-\ ] For any $i$, $\Omega^i_{X'/k}(\log D')[-i]$ is a complex with components $\Omega^i_{X'/k}(\log D')$ for degree $i$ and $0$ for other degrees. 
\item[-\ ]Let $f:K_1^*\to K^*_2$ be a morphism between complexes over some abelian category. Then we say $f$ is a quasi-isomorphism if $H^i(f):H^i(K_1^*)\to H^i(K_2^*)$ is an isomorphism for any $i$.
\item[-\ ]$\iota$ is the natural augmentation which is a quasi-isomorphism (see \cite[Lemma 4.2, Chapter III]{Hartshorne}).
\item[-\ ] $\varphi_{X}$ is an $\mathcal{O}_{X'}$-linear morphism of complexes constructed below, which is a quasi-isomorphism (see \cite[Theorem 2.1]{DI}).
\end{itemize} 
\begin{construction}\label{construction varphi}
Let $\{\zeta_{\alpha_0}\},\{h_{\alpha_0\alpha_1}\}$ be as in Lemma \ref{key lemma}. The morphism $\varphi_{X}$ in \eqref{varphi} is constructed as follows:
\begin{itemize}
\item[-\ ]$\varphi_{X}^0$ is clear;
\item[-\ ] $\varphi_{X}^i=(\varphi_{X}^1)^{\cup i}\delta_i$ for $0<i<p$, where $\delta_i$ is the standard section of the natural projection
$\Omega^{\otimes i}_{X'/k}(\log D')\to\Omega^i_{X'/k}(\log D')$ and
$$(\varphi_{X}^1)^{\cup i}:\Omega^{\otimes i}_{X'/k}(\log D')\to\check{\mathcal{C}}(\mathcal{U}',\tau_{<p}F_{X/k*}\Omega^*_{X/k}(\log D))^i$$
is given by
$$(\varphi_{X}^1)^{\cup i}(\omega_1\otimes\cdots\otimes\omega_i)=\varphi^1_{X}(\omega_1)\cup\cdots\cup\varphi^1_{X}(\omega_i),~\omega_1,\cdots,\omega_i\in\Omega_{X/k}(\log D).$$

\item[-\ ] $\varphi_{X}^1=\varphi_{X}(0,1)\oplus\varphi_{X}(1,0)$, where
$$\varphi_{X}(0,1):\Omega_{X'/k}(\log D')\to\check{\mathcal{C}}^0(\mathcal{U}',F_{X/k*}\Omega_{X/k}(\log D))$$
is given by $\varphi_{X}(0,1)(\omega)=\{\zeta_{\alpha}(\omega)\}$, and
$$\varphi_{X}(1,0):\Omega_{X'/k}(\log D')\to\check{\mathcal{C}}^1(\mathcal{U}',F_{X/k*}\mathcal{O}_X)$$
is given by $\varphi_X(1,0)(\omega)=\{h_{\alpha_0\alpha_1}(\omega)\}$.
\end{itemize}
\end{construction}
To understand Theorem \ref{D-I, decomposition}, it remains to recall $D(X')$.
\begin{definition}[{{\cite[\S4]{IL02}}}]
Let $\mathbb A$ be an abelian category. We define $C^+(\mathbb A)$ to be the category of bounded below complexes over $\mathbb A$. Set $K^+(\mathbb A)$ to be the homotopy category obtained by $C^+(\mathbb A)$ modulo the homotopy equivalence. Let $D^+(\mathbb A)$ be the derived category obtained by inverting quasi-isomorphisms in $K^+(\mathbb A)$. Set $D^+(X'):=D^+(\mathrm{Mod}_{\mathcal{O}_{X'}})$, where $\mathrm{Mod}_{\mathcal{O}_{X'}}$ is the abelian category of $\mathcal{O}_{X'}$-modules on $X'$.
\end{definition}

\subsection{Filtered derived category} For a thorough treatment of this subject, we refer the reader to \cite[\S3.3.1]{CEGT} or \cite[\S7]{D3}. Let $\mathbb{A}$ be an abelian category.
A filtered object of $\mathbb A$ is an object of $\mathbb A$ together with a finite increasing filtration. A filtered morphism in $\mathbb A$ is a morphism $f:(A_1,W^1)\to(A_2,W^2)$ between two filtered objects of $\mathbb A$, i.e., $f:A_1\to A_2$ is a morphism of $\mathbb A$ subjects to $f(W^1_iA_1)\subset W^2_iA_2$ for any $i$. Denote by $F\mathbb A$ the category of filtered objects of $\mathbb A$ and filtered morphisms in $\mathbb A$. It is an additive category.

Let $C^+(F\mathbb A)$ be the category of bounded below complexes over $F\mathbb A$, and let $K^+(F\mathbb A)$ be the homotopy category obtained by $C^+(F\mathbb A)$ modulo the homotopy equivalence. A morphism $f:(A_1^*,W^1)\to(A_2^*,W^2)$ in $C^+(F\mathbb A)$ is said to be a quasi-isomorphism if $f:A_1^*\to A_2^*$ is a filtered quasi-isomorphism with respect to the filtrations $W^1,W^2$, i.e., the induced morphism $f:W^1_iA_1^*\to W_i^2A_2^*$ is a quasi-isomorphism in $C^+(\mathbb A)$ for any $i$. Let $D^+(F\mathbb A)$ be the derived category obtained by inverting quasi-isomorphisms in $K^+(F\mathbb A)$.

 \subsection{The weight filtration} Let $S$ be a scheme, $X$ a smooth $S$-scheme, $D$ a normal crossing divisor on $X$ relative to $S$. Recall that besides the decreasing Hodge filtration $\Fil$ on the logarithmic de Rham complex $\Omega^*_{X/S}(\log D)$, it carries another increasing weight filtration $W$.
\begin{definition}[{{\cite[\S3.4.1.2]{CEGT}}}]\label{Deligne's weight filtration}
For any $i\geq0$, there is an increasing weight filtration $W$ on $\Omega^i_{X/S}(\log D)$ defined as follows:
$$W_l\Omega_{X/S}^i(\log D):=\left\{\begin{matrix}\Omega^l_{X/S}(\log D)\wedge\Omega^{i-l}_{X/S},&l\leq i;\\
\Omega_{X/S}^i(\log D),&l>i.
\end{matrix}
\right.$$
The weight filtration $W$ on each term $\Omega^i_{X/S}(\log D)$ induces a weight filtration on $\Omega^*_{X/S}(\log D)$ or $\bigoplus_i\Omega^i_{X/S}(\log D)[-i]$, again denoted by $W$.\end{definition}

\subsection{ $\varphi_{X}$ is an isomorphism in $D^+F(X')$} 
Let $(X,D)/k,(\tilde X,\tilde D)/W_2(k),\varphi_X$ be as in \S2.1. 
 Set 
$$C^+F(X'):=C^+(F\mathrm{Mod}_{\mathrm{O}_{X'}}),~K^+F(X'):=K^+(F\mathrm{Mod}_{\mathrm{O}_{X'}})$$
and $D^+F(X'):=D^+(F\mathrm{Mod}_{\mathrm{O}_{X'}})$. Note that the weight filtration $W$ on $\Omega^*_{X/k}(\log D)$ induces a filtration on the complex $\tau_{<p}F_{*}\Omega^*_{X/k}(\log D)$ or $\check{\mathcal{C}}(\mathcal{U}',\tau_{<p}F_{*}\Omega^*_{X/k}(\log D))$, again denoted by $W$. By definition the following
 $$(\tau_{<p}F_{*}\Omega^*_{X/k}(\log D),W),~(\check{\mathcal{C}}(\mathcal{U}',\tau_{<p}F_{*}\Omega^*_{X/k}(\log D)),W),~(\bigoplus^{p-1}_{i=0}\Omega^i_{X'/k}(\log D')[-i],W)$$
 are objects of $C^+F(X')$. It is obvious that $\iota$ is a quasi-isomorphism in $C^+F(X')$ and $\varphi_X$ is a morphism in $C^+F(X')$. Moreover, we have
\begin{proposition}\label{varphi quasi-iso} $\varphi_X$ is an isomorphism in $D^+F(X')$.
\end{proposition}
We prove this proposition by four steps. Since the problem is \'etale local, we may assume that $D$ is a simple normal crossing divisor on $X$ and $F$ admits a global log lifting $\tilde F$. Write $D=\sum_{i=1}^mD_i$, where each $D_i$ is a smooth component of $D$. For any $\emptyset\neq I\subset\{1,\cdots,m\}$, set $ D_I:=\cap_{i\in I} D_i$. Set $ D_\emptyset=X$.  

{\bf Step 1}. Recall that $\varphi_X$ is constructed by a D-I atlas $\{U_\alpha,\tilde F_{U_\alpha}\}$ of $(\tilde X,\tilde D)/W_2(k)$. Using the global log lifting $\tilde F$, there is another D-I atlas $\{U_\alpha,\tilde F|_{\tilde U_\alpha}\}$ of $(\tilde X,\tilde D)/W_2(k)$. Running Construction \ref{construction varphi} for this new D-I atlas, we obtain a new morphism which is denoted by $\varphi_{\tilde F}$. 
\begin{lemma}
$\varphi_X=\varphi_{\tilde F}$ in $K^+F(X')$. 
\end{lemma}
\begin{proof}
This lemma is a special case of Lemma \ref{homotopy formula}: we take $(\tilde Y,\tilde E)=(\tilde X,\tilde D),\tilde f=id_{(\tilde X,\tilde D)}$, D-I atlases $\{U_\alpha,\tilde F_{U_\alpha}\},\{U_\alpha,\tilde F|_{\tilde U_\alpha}\}$ of $(\tilde X,\tilde D)/W_2(k),(\tilde Y,\tilde E)/W_2(k)$, respectively and $\chi=\mathrm{id}$.
\end{proof}
This lemma shows that if $\varphi_{\tilde F}$ is a quasi-isomorphism in $C^+F(X')$, then Proposition \ref{varphi quasi-iso} follows. Let
 $$\zeta_{\tilde F}:\Omega_{X'/k}(\log D')\to F_*\Omega_{X/k}(\log D)$$
 be the morphism induced by $\tilde F$. Then we can construct a morphism

$$\tilde \varphi_{\tilde F}:(\bigoplus_i\Omega^i_{X'/k}(\log D')[-i],W)\to(F_*\Omega^*_{X/k}(\log D),W)$$
 in $C^+F(X')$ as follows: for $\omega_1,\cdots,\omega_s\in\Omega_{X'/k}(\log D')$, we set
 $$\tilde \varphi_{\tilde F}(\omega_1\wedge\cdots\wedge\omega_s):=\zeta_{\tilde F}(\omega_1)\wedge\cdots\wedge\zeta_{\tilde F}(\omega_s);$$
 for $x'\in\mathcal{O}_{X'}$, we set $\tilde\varphi_{\tilde F}(x'):=F^*x'$. Let
 $$\tau_{<p}\tilde\varphi_{\tilde F}:(\bigoplus_{i=0}^{p-1}\Omega^i_{X'/k}(\log D')[-i],W)\to(\tau_{<p}F_*\Omega^*_{X/k}(\log D),W)$$
 be the morphism obtained by truncating $\tilde\varphi_{\tilde F}$ at degrees $<p$, then one can check that $\varphi_{\tilde F}=\iota\circ \tau_{<p}\tilde\varphi_{\tilde F}$ holds in $C^+F(X')$.  
 Since $\iota$ is a quasi-isomorphism in $C^+F(X')$, it is enough to show that
\begin{lemma}\label{tilde varphi iso} $\tau_{<p}\tilde \varphi_{\tilde F}$ is an isomorphism in $D^+F(X')$.\end{lemma}

{\bf Step 2}. When take grading $\Gr^W_l$ on $\tau_{<p}F_*\Omega^*_{X/k}(\log D)$, the truncation $\tau_{<p}$ may make trouble. the following lemma dispel this worry.
 \begin{lemma}\label{quasi-iso truncation case}
  The natural morphism
 \begin{eqnarray}\label{com truncation}
\mu:Gr^W_l\tau_{<p}F_{*}\Omega^*_{X/k}(\log D)\to\tau_{<p}\Gr^{W}_lF_{*}\Omega^*_{X/k}(\log D)
  \end{eqnarray}
  is a quasi-isomorphism in $C^+(\mathrm{Mod}_{\mathcal O_{X'}})$.
 \end{lemma} 
 \begin{proof}
 It is enough to check that $\mathcal{H}^{p-1}(\mu)$ is an isomorphism. Since the problem is Zariski local, we may assume that $X$ admits a global coordinate system $t_1,\cdots,t_n$ such that $D=(t_1\cdots t_m=0)$. A direct computation shows that
 \begin{eqnarray*}
 \begin{array}{cl}
 &\mathcal{H}^{p-1}(\frac{W_{l}\tau_{<p}F_{*}\Omega^*_{X/k}(\log D)}{W_{l-1}\tau_{<p}F_{*}\Omega^*_{X/k}(\log D)})\\
 =&\frac{\{x\in F_{*}W_{l}\Omega^{p-1}_{X/k}(\log D):~dx=0\}}{\{x\in F_{*}W_{l-1}\Omega^{p-1}_{X/k}(\log D):~dx=0\}+dW_{l}\Omega^{p-2}_{X/k}(\log D)}
 \end{array}
 \end{eqnarray*}
 and
 \begin{eqnarray*}
 \begin{array}{cl}
 &\mathcal{H}^{p-1}(\Gr^{W}_lF_{*}\Omega^*_{X/k}(\log D))\\
 =&\frac{\{x\in F_{*}W_{l}\Omega^{p-1}_{X/k}(\log D):~dx\in F_{*}W_{l-1}\Omega^p_{X/k}(\log D)\}}{F_{*}W_{l-1}\Omega^{p-1}_{X/k}(\log D)+dF_{*}W_{l}\Omega^{p-2}_{X/k}(\log D)}.
 \end{array}
 \end{eqnarray*}
 The injectivity of $\mathcal{H}^{p-1}(\mu)$ is clear. Let us show the surjectivity of $\mathcal{H}^{p-1}(\mu)$. For any $\emptyset\neq I=\{i_1,\cdots,i_s\}\subset \{1,\cdots,m\}$ with $i_1<\cdots<i_s$, write 
 $$t_I=\prod_{i\in I}t_j,~d_I\log t=d\log t_{i_1}\wedge\cdots \wedge d\log t_{i_s}.$$
 Let $ x\in F_{*}W_{l}\Omega^{p-1}_{X/k}(\log D)$, then it can be locally expressed as the following form:
 $$x=\sum_{|I|=l}d_I\log t\wedge x_I+x',~x_I\in\Omega^{p-1-l}_{X/k},~x'\in W_{l-1}\Omega^{p-1}_{X/k}(\log D).$$
 Assume that $dx\in W_{l-1}\Omega^p_{X/k}(\log D)$. By Lemma \ref{decomposition varphi} (i), we have $d_{D_I/k}(x|_{D_I})=0$ for any $|I|=l$. Using the Cartier isomorphism \cite[Theorem 9.14]{EV}, we locally have 
 $$x_I|_{D_I}=\sum_{K\cap I=\emptyset,|K|=p-1-l}(\lambda_K^pt_K^pd_K\log t)|_{D_I}+d_{D_I/k}(x_I'|_{D_I}),~\lambda_K\in\mathcal{O}_X,~x_I'\in\Omega^{p-2-l}_{X/k}.$$
 Consequently, we locally get
 \begin{eqnarray*}x=\sum
 \lambda_K^pd_I\log t\wedge t_K^pd_K\log t+dx''+x''',\end{eqnarray*}
 where $x''\in W_{l}\Omega^{p-2}_{X/k}(\log D),x'''\in W_{l-1}\Omega^{p-1}_{X/k}(\log D)
$. From which the surjectivity of $\mathcal{H}^{p-1}(\mu)$ follows, this completes the proof.
  \end{proof}
  Combing this lemma with $5$-lemma, it is enough to show that
  \begin{lemma}\label{Cartier iso}$\Gr^W_l\tilde\varphi_{\tilde F}$ is an isomorphism in $D^+(X')$ for any $l$,or equivalently, $\tilde \varphi_{\tilde F}$ is an isomorphism in $D^+F(X')$. \end{lemma}

{\bf Step 3}. 
We inteoduce the Poincar\'e residue map.
\begin{definition}
(i) Let $I\subset\{1,\cdots,m\}$, and $i_{D_I}:D_I\hookrightarrow X$ be the natural closed embedding. For any $0\leq l\leq q$ with $|I|=l$, we define a poincar\'e residue map (\cite[\S3.4.1.3]{CEGT})
\begin{eqnarray*}\label{residue map}
\mathrm{Res}_{D_I}:W_l\Omega^q_{X/k}(\log D)\to i_{D_I*}\Omega^{q-l}_{D_I/k}
\end{eqnarray*}
as follows. Let $I=\{i_1,\cdots,i_l\}$ with $i_1<\cdots<i_l$, and let $\omega\in W_l\Omega^q_{X/k}(\log D)$. Locally, we can write 
$$\omega=d\log t_1\wedge\cdots\wedge d\log t_l\wedge\omega'+\eta,$$
where $t_1,\cdots,t_l$ are local defining functions of $D_{i_1},\cdots,D_{i_l}$ and 
$$\omega'\in\Omega^{q-l}_{X/k},~\eta\in\sum_{i\in I}\Omega^q_{X/k}(\log D-D_i).$$ Then we locally define 
$$\mathrm{Res}_{D_J}(\omega)=\omega'|_{D_J}\in i_{D_I*}\Omega^{q-l}_{D_I/k}.$$

(ii) Let $(K^*,d)$ be a complex over abelian category $A$, and let $l$ be an integer. We define $(K^*,d)[l]$ as follows:
$$(K^*,d)[l]:=(K^*[l],d[l])~,K^*[l]^i:=K^{i+l},~d[l]^i:=(-1)^ld^{i+l}.$$
\end{definition}
\begin{lemma}\label{decomposition varphi}
 Let $i_{D_I'}:D_I'\hookrightarrow X'$ be the base change of $i_{D_I}:D_I\hookrightarrow X$ via $F_k$. Then we have the following decompositions:
\begin{itemize}
\item[(i)\ ] $\Gr_l^{W}\Omega^*_{X/k}(\log D)\cong\bigoplus_{|I|=l}i_{D_I*}\Omega^*_{D_I/k}[-l]$;
\item[(ii)\ ] $\Gr^{W}_l\bigoplus_q\Omega^q_{X'/k}(\log D')[-q]\cong\bigoplus_{|I|=l}i_{D'_I*}(\bigoplus_{q}\Omega^q_{D'_I/k}[-q])[-l]$.
\end{itemize}
\end{lemma}

 \begin{proof}
 It is easy to see that $\mathrm{Res}_{D_I}$ induces a morphism
 $$
 Gr^W_{|I|}\Omega^q_{X/k}(\log D)\to\Omega^{q-|I|}_{D_I},
  $$
  again denoted by $\mathrm{Res}_{D_I}$.
 Moreover, one can check that
 $$\bigoplus_{|I|=l}\mathrm{Res}_{D_I}:Gr^W_l\Omega^q_{X/k}(\log D)\to\bigoplus_{|I|=l}i_{D_I*}\Omega^{q-l}_{D_I/k}$$
 is an isomorphism, from which (i) follows. (ii) is similar to (i).
 \end{proof}
  
 {\bf Step 4}. We study the mixed structure of $\tilde \varphi_{\tilde F}$.  Set $\tilde F_I$ to be the restriction of $\tilde F$ on $\tilde D_I$, then it induces a morphism
 $$\zeta_{\tilde F_I}:\Omega_{D_I'/k}\to F_*\Omega_{D_I/k}.$$
 We similarly define $\tilde\varphi_{\tilde F_I}$ as $\varphi_{\tilde F}$.
 For any $|I|=l$ and $i\geq l$, the following diagram can be easily checked:
 $$
\xymatrix{W_l\Omega^i_{X'/k}(\log D)\ar[d]^-{\mathrm{Res}_{D'_I}}\ar[r]^-{\tilde\varphi_{\tilde F}}&W_lF_*\Omega^i_{X/k}(\log D)\ar[d]^-{\mathrm{Res}_{D_I}}\\
i_{D'_I*}\Omega^{i-l}_{D'_I/k}\ar[r]^-{\tilde\varphi_{\tilde F_I}}&i_{D'_I}F_*\Omega^{i-l}_{D_I/k}.
} 
 $$
 Combing Lemma \ref{decomposition varphi} with this diagram, we get
 $$\Gr^W_l\tilde\varphi_{\tilde F}=\bigoplus_{|I|=l}i_{D'_I*}\tilde\varphi_{\tilde F_I}[-l].$$
 By the Cartier isomorphism {\itshape loc.cit.}, we obtain that $\tilde\varphi_{\tilde F_I}$ is a quasi-isomorphism in $C^+F(D_I')$ for any $I$. This implies that $\Gr^W_l\tilde \varphi_{\tilde F}$ is an isomorphism in $D^+(X')$, from which Lemma \ref{Cartier iso} follows. The proof of Proposition \ref{varphi quasi-iso} is completed.

As an easy consequence of Proposition \ref{varphi quasi-iso},  Theorem \ref{FSDT} follows.

Using Proposition \ref{varphi quasi-iso} and an argument which is similar to the proof of Raynaud's vanishing theorem \cite[Theorem 5.8]{IL02}, we have the following.
\begin{theorem}\label{filtered vanishing p}
Let $X$ be a smooth projective variety over $k$, and let $D$ be a normal crossing divisor on $X$. Assume that $\dim(X)<p$. Then for any ample line bundle $L$ on $X$, we have
\begin{eqnarray}
\mathbb{H}^j(X,W_l\Omega^i_{X/k}(\log D)\otimes L)=0~\mathrm{for}~i+j>\dim(X).
\end{eqnarray}
\end{theorem}
Note that when $D$ is a simple normal crossing divisor on $X$, it is an easy consequence of \cite[Theorem 5.8]{IL02}. 

Combing Theorem \ref{filtered vanishing p} with spreading-out technical (see \cite[\S6]{IL02}), Theorem \ref{vanishing theorem} follows.

\section{Proof of Theorem \ref{main theorem}}
This section aims to the proof of Theorem \ref{main theorem}. 
\subsection{Mixed Fontaine-Lafffaille complexes}
In this subsection, we establish a positive characteristic analogue of the notion of mixed Hodge complexes (see \cite[Definition 3.3.17]{CEGT}, \cite[(8.1.5)]{D3}).
\begin{definition}
We call a quadruple $(K^*_\dR,W,\Fil,\psi)$ a mixed Fontaine-Laffaille complex over $k$ if:
\begin{itemize}
\item[-\ ] $K^*_\dR$ is a bounded below complex of $k$-vector spaces;
\item[-\ ] $W$ is a finite increasing weight filtration on $K^*_\dR$;
\item[-\ ] $\Fil$ is a finite decreasing Hodge filtration on $K^*_\dR$;
\item[-\ ] $H^j(\Gr_i^WK^*_\dR)=0$ for $i+j\gg0$ and it is a finite dimensional $k$-vector space for any $i,j$;
\item[-\ ] let $(K^*_\Hig,W):=F_k^*\Gr_\Fil(K^*_\dR,W)$, then 
\begin{eqnarray}\label{psi FLC}
\psi:(K_\Hig^*,W)\to(K^*_\dR,W)\end{eqnarray}
is an isomorphism in $D^+F(\mathrm{Vec}_k)$.
\end{itemize}
A morphism between two mixed Fontaine-Laffaille complexes
$$f:(K^*_{\dR,1},W^1,\Fil_1,\psi_1)\to(K^*_{\dR,2},W^2,\Fil_2,\psi_2)$$
is a morphism of complexes of $k$-vector spaces 
$$f:K^*_{\dR,1}\to K^*_{\dR,2}$$ 
satisfying the following conditions:
\begin{itemize}
\item[-\ ] $f$ preserves the filtrations $W^1,W^2$ and $\Fil_1,\Fil_2$;
\item[-\ ] the following diagram is commutative in $D^+F(\mathrm{Vec}_k)$:
\begin{eqnarray}\label{FL morphism}
\xymatrix{(K^*_{\Hig,1},W^1)\ar[r]^-{F_k^*\Gr(f)}\ar[d]^-{\psi_1}&(K^*_{\Hig,2},W^2)\ar[d]^-{\psi_2}\\
(K^*_{\dR,1},W^1)\ar[r]^-{f}&(K^*_{\dR,2},W^2).}
\end{eqnarray}
\end{itemize}
Denote by $\mathrm{MFLC}(k)$ the category of  mixed Fontaine-Laffaille complexes over $k$.
\end{definition}
\begin{remark} When $K_\dR^0=V$, $K_\dR^i=0$ for $i\neq0$ and $W_{-1}K^*_\dR=0,W_0K_\dR^*=K_\dR^*$, the mixed Fontaine-Laffaille complex $(K^*_\dR,W,\Fil,\psi)$ becomes a Fontaine-Laffaille module $(V,\Fil,\psi)$ over $k$. In this case, we notice that $\psi:F_k^*\Gr_\Fil V\to V$ and \eqref{FL morphism} live in $\mathrm{Vec}_k$. Denote by $\mathrm{MF}(k)$ the category of Fontaine-Laffaille modules over $k$.
\end{remark}
It is well-known that the category of Hodge structures is an abelian category and the morphisms are strict, i.e., if $f:(H_1,\Fil_1)\to (H_2,\Fil_2)$ is a morphism of Hodge structures, then 
$$f(\Fil^iH_1)=f(H_1)\cap \Fil^iH_2.$$
For the category $\mathrm{MF}(k)$, we have similar properties.
\begin{theorem}[{{\cite[Theorem 2.1]{Fa}}}]\label{MF abelian}
$\mathrm{MF}(k)$ is an abelian category and morphisms in $\mathrm{MF}(k)$ are strict.
\end{theorem}
In what follows, let us fix a mixed Fontaine-Laffaille complex $(K^*_\dR,W,\Fil,\psi)$ and study its spectral sequences. We briefly recall the relevant notions.

\begin{definition}
Let $\mathbb A$ be an abelian category. For any $r_0\geq 0$, a spectral sequence starting from $E_{r_0}$-page over $\mathbb A$ is a family $\{E_r^{i,j},d_r^{i,j}\}_{r\geq r_0}$ subjects to the following conditions:
\begin{itemize}
\item[-\ ] for any $i,j$ and $r\geq r_0$, $E^{i,j}_r$ is an object of $\mathbb A$ and $d^{i,j}_r:E_r^{i,j}\to E_r^{i+r,j-r+1}$ is a morphism in $\mathbb A$;
\item[-\ ] $d_r^{i,j}\circ d^{i-r,j+r-1}_r=0$ and $E_{r+1}^{i,j}\cong\frac{\mathrm{ker}(d_r^{i,j})}{\mathrm{im}(d_r^{i-r,j+r-1})}$.
\end{itemize}
 A morphism between spectral sequences starting from $E_{r_0}$-page over $\mathbb A$
$$f:\{E_r^{i,j},d_r^{i,j}\}_{r\geq r_0}\to\{E_r'^{i,j},d_r'^{i,j}\}_{r\geq r_0}$$
consisting of the following data:
\begin{itemize}
\item[-\ ] for any $i,j$ and $r\geq r_0$, there is a morphism $f:E_r^{i,j}\to E_r'^{i,j}$ in $\mathbb A$ which fits into a commutative diagram
$$\xymatrix{E_r^{i,j}\ar[r]^-{d_r^{i,j}}\ar[d]^-{f}&E_r^{i+r,j-r+1}\ar[d]^-{f}\\
E_r'^{i,j}\ar[r]^-{d_r'^{i,j}}&E_r'^{i+r,j-r+1};
}
$$
\item[-\ ] $f:E_{r+1}^{i,j}\to E_{r+1}'^{i,j}$ is induced from the $E_r$-page.
\end{itemize}
Denote by $E_{r\geq r_0}/\mathbb A$ the category of spectral sequences starting from $E_{r_0}$-page over $\mathbb A$, which is an additive category.
\end{definition}
A typical example of spectral sequences is the following.
\begin{example}[{{\cite[page 440]{GH}}}]
 Let $K^*$ be a complex over an abelian category $\mathbb A$  with differential $d$ and $F$ a  decreasing filtration on $K^*$.  We construct the spectral sequence $\{E_r^{i,j}(K^*,W),d_r^{i,j}\}_{r\geq 0}$ of $(K^*,F)$ starting from $E_0$-page as follows:
$$
\begin{array}{c}
E^{i,j}_r(K^*,W)=Z^{i,j}_r/B^{i,j}_r,~Z^{i,j}_r=\{x\in F^iK^{i+j}:dx\in F^{i+r}K^{i+j+1}\},\\
B^{i,j}_r=dF^{i-r+1}K^{i+j-1}\cap F^iK^{i+j}+\{x\in F^{i+1}K^{i+j}:dx\in F^{i+r}K^{i+j+1}\}
\end{array}
$$
and $d_r^{i,j}:E^{i,j}_r\to E^{i+r,j-r+1}_r$ is induced by $d$. Note that we can take $r=\infty$. We say this spectral sequence degenerates at $E_r$ if $d_s^{i,j}=0$ for any $i,j$ and $s\geq r$.

Let $W$ be another finite increasing filtration on $K^*$. Set $\tilde W^i:=W_{-i}$, Then we take the spectral sequence of $(K^*,W)$ to be the spectral sequence of $(K^*,\tilde W)$.
\end{example}
\begin{lemma}\label{W E_1}
 The spectral sequences of $(K^*_\dR,\Fil),\Gr^W_i(K^*_\dR,\Fil)$ degenerate at $E_1$.
\end{lemma}
\begin{proof}
It is obvious that
$$\mathrm{dim}_k~H^m(K^*_\dR)=\mathrm{dim}_k~H^m(F_k^*\Gr_\Fil K^*_\dR)=\mathrm{dim}_k~H^m(\Gr_\Fil K^*_\dR),$$
from which the $E_1$-degeneration of $(K^*_\dR,\Fil)$ follows. Note that 
$$F_k^*\Gr_\Fil\Gr^W_iK_\dR^*=\Gr^W_iF_k^*\Gr_\Fil K_\dR^*,$$
then the $E_1$-degeneration of $\Gr^W_i(K^*_\dR,\Fil)$ follows similarly.
\end{proof}
We proceed to study the spectral sequence of $(K_\dR^*,W)$. Before doing this, it is necessary to review Deligne's technical result.
\begin{definition}[{{\cite[Definition 3.2.26]{CEGT}, \cite[(7.2.4)]{D3}}}]\label{three filtration}
Let $K^*$ be a bounded below complex over an abelian category $\mathbb A$ with differential $d$,  $W$ a finite increasing filtration and $F$ a finite decreasing filtration on $K^*$. Assume that $H^j(\Gr^W_i K^*)=0$ and $ H^j(\Gr_F^iK)=0$ for $i+j\gg0$. Then there are three decreasing filtrations $F_d,F_{d^*},F_{\mathrm{rec}}$ on $E^{i,j}_r(K^*,W)$ which are defined as follows:
\begin{itemize}
\item[-\ ] $F_d^lE^{i,j}_r(K^*,W):=\mathrm{im}(E_r^{i,j}(F^lK^*,W)\to E^{i,j}_r(K^*,W))$;
\item[-\ ] dually, we define 
$F_{d^*}^lE^{i,j}_r(K^*,W):=\mathrm{ker}(E_r^{i,j}(K^*,W)\to E^{i,j}_r(K^*/F^lK^*,W))$;
\item[-\ ] it is easy to check that $F_d=F_{d^*}$ on $E_0^{i,j}(K^*,W)$. Then we define the recurrent filtration $F_{\mathrm{rec}}$ on $E_r^{i,j}(K^*,W)$ by induction on $r$. Set $F_{\mathrm{rec}}:=F_d=F_{d^*}$ on $E_0^{i,j}(K^*,W)$. Assume $F_{\mathrm{rec}}$ is defined on $E_r^{i,j}(K^*,W)$. Since $E_{r+1}^{i,j}(K^*,W)$ is a subquotient of $E_r^{i,j}(K^*,W)$, then it induces a filtration $F_{\mathrm{rec}}$ on $E_{r+1}^{i,j}(K^*,W)$. 
\end{itemize}
\end{definition}
Now we can state Deligne's technical lemma.
\begin{lemma}[Deligne {{\cite[Theorem 3.2.30]{CEGT}, \cite[Proposition 7.2.5]{D3}}}]\label{Deligne filtration}
Use the notation as in the definition above.
\begin{itemize}
\item[(i)\ ] The differential $d$ is strictly compatible with $F$ if and only if the spectral sequence of $(K^*,F)$ degenerates at $E_1$.
\item[(ii)\ ] Let $r_0\geq0$. Assume that $d_r^{i,j}$ is strictly compatible with $F_{\mathrm{rec}}$ for any $i,j$ and $r<r_0$. Then we have the following exact sequence:
\begin{eqnarray*}
0\to E^{i,j}_{r_0}(F^lK^*,W)\to E^{i,j}_{r_0}(K^*,W)\to E^{i,j}_{r_0}(K^*/F^lK^*,W)\to0.
\end{eqnarray*}
\end{itemize}
\end{lemma}

Denote by $$\{E_{r,\dR}^{i,j}=\frac{Z_{r,\dR}^{i,j}}{B_{r,\dR}^{i,j}},~d_{r,\dR}^{i,j}\},~\{E_{r,\Hig}^{i,j}=\frac{Z_{r,\Hig}^{i,j}}{B_{r,\Hig}^{i,j}},~d_{r,\Hig}^{i,j}\}$$ the spectral sequences of $(K_\dR^*,W),(K^*_\Hig,W):=F_k^*\Gr_\Fil(K_\dR^*,W)$, respectively. They are related by the following theorem.
\begin{theorem}\label{spectral sequence}Let $(K^*_\dR,W,\Fil,\psi)$ be as above.
\begin{itemize}
\item[(i)\ ] $\psi$ induces an isomorphism of spectral sequences 
$$\psi:\{E_{r,\Hig}^{i,j},d_{r,\Hig}^{i,j}\}_{r\geq1}\to\{E_{r,\dR}^{i,j},d_{r,\dR}^{i,j}\}_{\geq1}.$$
\item[(ii)\ ] $\Fil_{d}=\Fil_{\mathrm{rec}}$ on $E_{r,\dR}^{i,j}$ and 
 $F_k^*\Gr_{\Fil_{\mathrm{rec}}}$ sends $d_{r,\dR}^{i,j}$ to $d_{r,\Hig}^{i,j}$ for $r\geq0$.
 \end{itemize}
 \end{theorem}
 \begin{proof}
 (i) is obvious. For (ii), we restate it in a more subtle form, i.e., for any $i,j$ and $r\geq0$, we have:
 \begin{itemize}
 \item[-\ ] $\Fil_{d}=\Fil_{\mathrm{rec}}=\Fil_{d^*}$ on $E_{r,\dR}^{i,j}$ and $d^{i,j}_{r,\dR}$ is strict with respect to $\Fil_{\mathrm{rec}}$;
 \item[-\ ] there is an isomorphism 
 $$\mu:F_k^*\Gr_{\Fil_{\mathrm{rec}}}E^{i,j}_{r,\dR}\to E^{i,j}_{r,\Hig}$$
 in $\mathrm{Vec}_k$ such that the following diagram is commutative:
 \begin{eqnarray}\label{mu diagram}
 \xymatrix{
 \bigoplus_lF_k^*\Fil^lZ^{i,j}_{r,\dR}\ar[r]\ar[d]&Z^{i,j}_{r,\Hig}\ar[d]\\
 F_k^*\Gr_{\Fil_{\mathrm{rec}}}E^{i,j}_{r,\dR}\ar[r]^-{\mu}&E^{i,j}_{r,\Hig},
  }
 \end{eqnarray}
 where the upper and two vertical morphisms are natural morphisms. Moreover, the two vertical maps are surjective, from which $\mu$ is uniquely determined.  
 \item[-\ ]the following diagram is commutative:
 \begin{eqnarray}\label{mu commu}
 \xymatrix{
 F_k^*\Gr_{\Fil_{\mathrm{rec}}}E^{i,j}_{r,\dR}\ar[d]_-{F_k^*\Gr_{\Fil_{\mathrm{rec}}}d_{r,\dR}^{i,j}}\ar[r]^-{\mu}&E^{i,j}_{r,\Hig}\ar[d]^-{d_{r,\Hig}^{i,j}}\\
 F_k^*\Gr_{\Fil_{\mathrm{rec}}}E^{i+r,j-r+1}_{r,\dR}\ar[r]^-{\mu}&E^{i+r,j-r+1}_{r,\Hig}.
}
 \end{eqnarray}
 \end{itemize}
 We prove this subtle form by induction on $r$. The case $r=0$. According to Lemma \ref{W E_1} and Lemma \ref{Deligne filtration} (i), $d^{i,j}_{0,\dR}$ is strict with respect to $\Fil_{\mathrm{rec}}$ for any $i,j$. We set 
 $$\mu:F_k^*\Gr_{\Fil_{\mathrm{rec}}}\Gr^W_{-i}K^{i+j}_{\dR}\to\Gr^W_{-i}F_k^*\Gr_{\Fil_{\mathrm{rec}}}K^{i+j}_{\dR}=\Gr^W_{-i}K^{i+j}_\Hig$$
to be the natural isomorphism. The other statements are obvious, hence this case follows. 

Assume this subtle form holds for any $r\leq r_0$.  Observing that the isomorphism $\psi:(K_\Hig^*,W)\to(K_\dR^*,W)$ induces an isomorphism $E^{i,j}_{r,\Hig}\to E^{i,j}_{r,\dR}$ in $\mathrm{Vec}_k$ for any $i,j$ and $r\geq1$, again denoted by $\psi$. By Theorem \ref{MF abelian}, we know that
 \begin{eqnarray}\label{morphism in MF} d_{r,\dR}^{i,j}:(E_{r,\dR}^{i,j},\Fil_{\mathrm{rec}},\psi)\to(E_{r,\dR}^{i+r,j-r+1},\Fil_{\mathrm{rec}},\psi)\end{eqnarray}
 is a morphism in $\mathrm{MF}(k)$ for any $i,j$ and $1\leq r\leq r_0$. It follows that
 $$F_k^*\Gr_{\Fil_{\mathrm{rec}}}E^{i,j}_{r+1,\dR}\cong H(F_k^*\Gr_{\Fil_{\mathrm{rec}}}E_{r,\dR}^{i-r,j+r-1}\to F_k^*\Gr_{\Fil_{\mathrm{rec}}}E^{i,j}_{r,\dR}\to F_k^*\Gr_{\Fil_{\mathrm{rec}}}E^{i+r,j-r+1}_{r,\dR})
 $$
 holds for any $i,j$ and $1\leq r\leq r_0$. For any $i,j$,
let
 \begin{eqnarray}\label{mu r_0+1}\mu:F_k^*\Gr_{\Fil_{\mathrm{rec}}}E^{i,j}_{r_0+1,\dR}\to E_{r_0+1,\Hig}^{i,j}\end{eqnarray}
 be the cohomology of
 $$
 \xymatrix{
 F_k^*\Gr_{\Fil_{\mathrm{rec}}}E_{r,\dR}^{i-r,j+r-1}\ar[r]\ar[d]^-{\mu}&F_k^*\Gr_{\Fil_{\mathrm{rec}}}E^{i,j}_{r,\dR}\ar[r]\ar[d]^-{\mu}&F_k^*\Gr_{\Fil_{\mathrm{rec}}}E^{i+r,j-r+1}_{r,\dR}\ar[d]^-{\mu}\\
E_{r,\Hig}^{i-r,j+r-1}\ar[r]&E^{i,j}_{r,\Hig}\ar[r]&E^{i+r,j-r+1}_{r,\Hig}.}
 $$
One can check that \eqref{mu r_0+1} satisfies \eqref{mu diagram} for $r=r_0+1$. By assumption, $d^{i,j}_{r,\dR}$ is strict for any $i,j$ and $r\leq r_0$. Using Lemma \ref{Deligne filtration} (ii), it implies that $\Fil_{d}=\Fil_{\mathrm{rec}}$ on $E_{r_0+1,\dR}^{i,j}$ for any $i,j$. From which the surjectivity of the vertical morphisms in \eqref{mu diagram} for $r=r_0+1$ follows. Combing the following easily checked commutative diagram
$$
 \xymatrix{
 \bigoplus_lF_k^*\Fil^lZ^{i,j}_{r_0+1,\dR}\ar[r]\ar[d]&Z^{i,j}_{r_0+1,\Hig}\ar[d]\\
\bigoplus_lF_k^*\Fil^lZ^{i+r,j-r+1}_{r_0+1,\dR}\ar[r]&Z^{i+r,j-r+1}_{r_0+1,\Hig}
  }
$$
with the surjectivity of $\mu:\bigoplus_lF_k^*\Fil^lZ^{i,j}_{r_0+1,\dR}\to F_k^*\Gr_{\Fil_{\mathrm{rec}}} E^{i,j}_{r_0+1,\dR}$, the commutativity of \eqref{mu commu} for $r=r_0+1$ follows. It remains to show the strictness of $d^{i,j}_{r_0+1,\dR}$ with respect to $\Fil_{\mathrm{rec}}$ for any $i,j$. This is easy, because the commutativity of \eqref{mu commu} for $r=r_0+1$ says that \eqref{morphism in MF} is a morphism in $\mathrm{MF}(k)$ for $r=r_0+1$. The proof is completed.
  \end{proof}
Note that $F_k^*$ induces an automorphism of $E_{r\geq 1}/\mathrm{Vec}_k$. 
As an application of the theorem above, one has
\begin{lemma-definition}\label{definition rho MFLC}
There are functors 
$$\rho^{\mathrm{MFLC}}_\dR,~\rho^{\mathrm{MFLC}}_\Hig:\mathrm{MFLC}(k)\to E_{r\geq 1}/\mathrm{Vec}_k$$
with
$$
\begin{array}{c}
\rho^{\mathrm{MFLC}}_\dR(K^*_\dR,W,\Fil,\psi):=\{E_r^{i,j}(K_\dR^*,W),d^{i,j}_{r,\dR}\}_{r\geq 1},\\
\rho^{\mathrm{MFLC}}_\Hig(K^*_\dR,W,\Fil,\psi):=\{E_r^{i,j}(K_\Hig^*,W),d^{i,j}_{r,\Hig}\}_{r\geq 1}.
\end{array}
$$
The  recursive filtration $\Fil_{\mathrm{rec}}$ on $E_r^{i,j}(K_\dR^*,W)$ induces a Hodge filtration $\Fil$ on $\rho^{\mathrm{MFLC}}_\dR$ with $\rho^{\mathrm{MFLC}}_\Hig=F_k^*\circ\Gr_\Fil\rho^{\mathrm{MFLC}}_\dR$.
Let $\psi:\rho^{\mathrm{MFLC}}_\Hig\to\rho^{\mathrm{MFLC}}_\dR$
be the isomorphism of functors induced by \eqref{psi FLC}. We call $(\rho^{\mathrm{MFLC}}_\dR,\rho^{\mathrm{MFLC}}_\Hig,\Fil,\psi)$ a Fontaine-Laffaille quadruple over $\mathrm{MFLC}(k)$.
\end{lemma-definition} 
\begin{proof}
Using Theorem \ref{spectral sequence}, we know that $d_{r,\dR}^{i,j}$ is strict with respect to $\Fil_\mathrm{rec}$ for any $i,j$ and $r\geq 0$. It follows that
$$\{\Fil^l_{\mathrm{rec}}E_{r,\dR}^{i,j},d_{r,\dR}^{i,j}\}_{r\geq 1},~\{\Gr^l_{\Fil_{\mathrm{rec}}}E_{r,\dR}^{i,j},d_{r,\dR}^{i,j}\}_{r\geq1}\in E_{r\geq 1}/\mathrm{Vec}_k,~\forall~l.$$
In other words, the Hodge filtration $\Fil$ on $\rho^{\mathrm{MFLC}}_\dR$ is well-defined. By Theorem \ref{spectral sequence} (ii), $\rho^{\mathrm{MFLC}}_\Hig=F_k^*\circ\Gr_\Fil\rho^{\mathrm{MFLC}}_\dR$ follows. The other statements are obvious.
\end{proof}
 \subsection{Cohomological mixed Fontaine-Laffaille complexes}
 We introduce a positive characteristic analogue of the notion of cohomological mixed Hodge complexes (see \cite[Definition 3.3.18]{CEGT}, \cite[(8.1.6)]{D3}).
 \begin{definition}
Let $X$ be a scheme over $k$. Denote by $\mathrm{Mod}_{k_X}$ (resp.~$\mathrm{Mod}_{k_{X'}}$) the abelian category of sheaves of $k$-vector spaces on $X$ (resp.~$X'$). We call a quadruple $(\mathcal{K}^*_\dR,W,\Fil,\Psi)$ a cohomological mixed Fontaine-Laffaille complex over $X$ if:
\begin{itemize}
\item[-\ ] $\mathcal{K}^*_\dR\in C^+(\mathrm{Mod}_{k_X})$;
\item[-\ ] $W$ is a finite increasing weight filtration on $\mathcal{K}^*_\dR$;
\item[-\ ] $\Fil$ is a finite decreasing Hodge filtration on $\mathcal{K}^*_\dR$;
\item[-\ ] $R^j\Gamma(X, \Gr^W_i\mathcal{K}^*_\dR)=0$ for $i+j\gg0$ and it has finite dimension over $k$ for any $i,j$;
\item[-\ ] set $(\mathcal{K}^*_\Hig,W):=\pi_{X/k}^*\Gr_\Fil(\mathcal{K}^*_\dR,W)\in C^+((\mathrm{Mod}_{k_{X'}})$, then 
$$\Psi:(\mathcal{K}_\Hig^*,W)\to F_*(\mathcal{K}^*_\dR,W)$$
is an isomorphism in $D^+F(\mathrm{Mod}_{k_{X'}})$.
\end{itemize}
\end{definition} 
\begin{lemma} 
Let $(\mathcal{K}^*_\dR,W,\Fil,\psi)$ be a cohomological mixed Fontaine-Laffaille complex over $X$.  Then 
$(R\Gamma(X,(\mathcal{K}_\dR^*,W,\Fil)),R\Gamma(X',\Psi))$
 is a mixed Fontaine-Laffaille complex over $k$.
\end{lemma}
\begin{proof}
We abbreviate $R\Gamma(X,\Fil)$ as $\Fil$. For the proof of this lemma, it is enough to notice that
$$F_k^*\Gr_\Fil R\Gamma(X,(\mathcal{K}^*_\dR,W))=R\Gamma(X',\pi^*_{X/k}\Gr_\Fil(\mathcal{K}^*_\dR,W))=R\Gamma(X',(\mathcal{K}^*_\Hig,W))$$
and $R\Gamma(X,(\mathcal{K}^*_\dR,W))=R\Gamma(X',F_*(\mathcal{K}^*_\dR,W))$.
\end{proof}
Taking account Theorem \ref{FSDT}, we immediately obtain the following.
\begin{theorem}
Notation as in Theorem \ref{FSDT} and additionally assume that $X$ is smooth proper over $k$ with dimension less than $p$. Then $(\Omega^*_{X/k}(\log D),W,\Fil,\Psi_{(\tilde X,\tilde D)})$ is a cohomological mixed Fontaine-Laffaille complex over $X$.
\end{theorem}
\subsection{The functoriality of $\Psi$}
Let $\tilde f:(\tilde X,\tilde D)\to(\tilde Y,\tilde E)$ be a morphism  in the category $\mathrm{Sch}^{\mathrm{sp}}_{W_2(k),\log}$, and let $ f:( X, D)\to( Y, E)$ be its closed fiber in $\mathrm{Sch}^s_{k,\log}$. Choosing D-I atlases  
$\{U_{\alpha},\tilde F_{U_{\alpha}/k}\}_{\alpha\in\mathcal{A}},\{V_{\beta},\tilde F_{V_{\beta}/k}\}_{\beta\in\mathcal{B}}$ 
of $(\tilde X,\tilde D)/W_2(k),(\tilde Y,\tilde E)/W_2(k)$, respectively, and set
$\mathcal{U}:=\{U_{\alpha}\}_{\alpha\in\mathcal{A}},\mathcal{V}:=\{V_{\beta}\}_{\beta\in\mathcal{B}}.$ 
By suitably choose $\mathcal{A},\mathcal{B}$, we may assume that there exists a map $\chi:\mathcal{A}\to\mathcal{B}$ such that $f(U_{\alpha})\subset V_{\chi(\alpha)}$ for any $\alpha\in\mathcal{A}$.
Note that the pullback 
\begin{eqnarray*}f^*:(\Omega^*_{Y//k}(\log E),W)\to f_{*}(\Omega^*_{X/k}(\log D),W)
\end{eqnarray*}
 together with $\chi$ induce a morphism in $C^+F(X')$:
\begin{eqnarray*}f^*:(\check{\mathcal{C}}(\mathcal{V}',F_*\Omega^*_{Y/k}(\log E)),W)\to f'_{*}(\check{\mathcal{C}}(\mathcal{U}',F_*\Omega^*_{X/k}(\log D)),W).
\end{eqnarray*}
The following lemma can be regarded as a sort of functoriality of $\Psi$.
 \begin{lemma}\label{homotopy formula}
 The following diagram is commutative in $D^+F(X')$:
 \begin{eqnarray*}\label{diagram XY}
\xymatrix{
 (\bigoplus_i\Omega^i_{Y'/k}(\log E')[-i],W)\ar[r]^-{f'^*}\ar[d]^-{\varphi_{Y}}\ar@/_6pc/[dd]_-{\Psi_{(\tilde Y,\tilde E)}}&f'_{*}(\bigoplus_i\Omega^i_{X'/k}(\log D')[-i],W)\ar[d]^-{\varphi_{X}}\ar@/^6pc/[dd]^-{\Psi_{(\tilde X,\tilde D)}}\\
(\check{\mathcal{C}}(\mathcal{V}',F_*\Omega^*_{Y/k}(\log E)),W)\ar[r]^-{f^*}&f'_{*}(\check{\mathcal{C}}(\mathcal{U}',F_*\Omega^*_{X/k}(\log D)),W)\\
(F_*\Omega^*_{Y/k}(\log E),W)\ar[u]_-{\iota}\ar[r]^-{f^*}&f'_*(F_*\Omega^*_{X/k}(\log D),W).\ar[u]_-{\iota}}
\end{eqnarray*}
\end{lemma}
 \begin{proof} 
 It is enough to show that the upper square is commutative in $K^+F(X')$. For this purpose, we need to construct a  morphism
\begin{eqnarray}\label{homotopy}
 \quad\eta_{i}:(\Omega^{\otimes i}_{Y'/k}(\log E'),W)\to f'_{*}(\check{\mathcal{C}}(\mathcal{U}',F_*\Omega^*_{X/k}(\log D))^{i-1},W)\end{eqnarray}
for any $0\leq i<p$ such that 
\begin{eqnarray*}
f^*\varphi^i_{Y}-\varphi_{X}^if'^*=d\eta_{i}\delta_{i}.
\end{eqnarray*}
 There we set $d$ to be the differential of $\check{\mathcal{C}}(\mathcal{U}',F_*\Omega^*_{X/k}(\log D))$, 
$$
\begin{array}{rcl}
W_j\Omega^{\otimes i}_{Y'/k}(\log E')&:=&W_{j}\Omega^{\otimes i}_{Y'/k}(\log E')\\
&:=&\left\{\begin{matrix}\sum_{\sigma\in S_i}\sigma\cdot\Omega^{\otimes (j)}_{Y'/k}(\log E')\otimes\Omega_{Y'/k}^{\otimes(i-j)},&j\leq i;\\
\Omega^{\otimes i}_{Y'/k}(\log E'),&j>i,
\end{matrix}
\right.
\end{array}$$
\footnote{The permutation group $S_i$ acts on $\Omega^{\otimes i}_{Y'/k}(\log E')$ as follows: $\sigma\cdot\omega_1\otimes\cdots\otimes\omega_i:=\mathrm{sgn}(\sigma)\cdot\omega_{\sigma(1)}\otimes\cdots\otimes\omega_{\sigma(i)}$ for $\omega_1,\cdots,\omega_i\in\Omega_{Y'/k}(\log E')$.} and $\delta_{i}$ the standard section of the projection $\Omega^{\otimes i}_{X'/k}(\log D')\to\Omega^i_{X'/k}(\log D')$.

We proceed to the construction of $\eta_i$. The case $i=0$ is obvious with $\eta_0=0$. Consider the case $i=1$. Choosing $\alpha\in\mathcal{A}$ and set $\beta=\chi(\alpha)$. Note that though the following diagram
\begin{eqnarray*}
\xymatrix{\tilde U_{\alpha}\ar[r]^-{\tilde F_{U_{\alpha}/k}}\ar[d]^-{\tilde f}&\tilde U'_{\alpha}\ar[d]^-{\tilde f'}\\
\tilde V_{\beta}\ar[r]^-{\tilde F_{V_{\beta}/k}}&\tilde V'_{\beta}}
\end{eqnarray*}
is not necessarily commutative at the level of schemes, but it is commutative at the level of topological spaces. Hence we can consider the difference
\begin{eqnarray*}
\tilde f^*\tilde F^*_{V_{\beta}/k}-\tilde F^*_{U_{\alpha}/k}\tilde f'^*:\mathcal{O}_{\tilde V'_{\beta}}\to p(\tilde{f}'\tilde{F}_{U_{\alpha}/k})_*\mathcal{O}_{\tilde U_{\alpha}}.
\end{eqnarray*}
Similar to the construction of $h_{\alpha_0\alpha_1}$ in Lemma \ref{key lemma}, the difference induces a morphism
\begin{eqnarray*}
\eta^{\alpha}_{1}:=(\tilde f^*\tilde F^*_{V_{\beta}/k}-\tilde F^*_{U_{\alpha}/k}\tilde f'^*)/p:\Omega_{Y'/k}(\log D')|_{V'_{\beta}}\to f_{*}'(F_{U_{\alpha}/k*}\mathcal{O}_{ U_{\alpha}}).
\end{eqnarray*}
 We define 
$$\eta_{1}:\Omega_{Y'/k}(\log E')\to f'_{*}\check{\mathcal{C}}^0(\mathcal{U}',F_{X/k*}\mathcal{O}_{X}),~\omega'\mapsto\{\eta^{\alpha}_{1}(\omega')\in\mathcal{O}_{U_{\alpha}}\}_{\alpha\in\mathcal{A}}.$$
One can check that $\eta_1$ gives rise to the claimed morphism in \eqref{homotopy} when $i=1$.
 Proceeding to the case $1<i<p$. Since 
$$
\begin{array}{rcl}
f^*(\varphi^1_{Y})^{\cup i}
&=&(f^*\varphi^1_{Y})^{\cup i}
=(\varphi_{X}^1f'^*+d\eta_{1})^{\cup i}\\
&=&(\varphi_{X}^1f'^*)^{\cup i}+d\sum^{i-1}_{j=0}(-1)^{j}(\varphi_{X}^1f'^*)^{\cup j}\cup\eta_{1}\cup(f^*\varphi^1_{Y})^{\cup(i-j-1)}, \end{array}
$$
where we set $(\varphi_{X}^1f'^*)^{\cup 0}=(f^*\varphi^1_{Y})^{\cup0}:=1$. Hence
\begin{eqnarray}\label{explicit homotopy}
\eta_{i}:=\sum^{i-1}_{j=0}(-1)^{j}(\varphi_{X}^1f'^*)^{\cup j}\cup\eta_{1}\cup(f^*\varphi^1_{Y})^{\cup(i-j-1)}\end{eqnarray}
is a claimed morphism in \eqref{homotopy}. This completes the proof.
\end{proof}
This lemma allows us to make the following definition. 
\begin{definition}
 There is a derived functor
$$R\Gamma:(\mathrm{Sch}^{\mathrm{sp}}_{W_2(k),\log})^{\mathrm{op}}\to\mathrm{MFLC}(k)$$
defined as follows: 
\begin{itemize}
\item[-\ ] let $(\tilde X,\tilde D)\in\mathrm{Sch}^{\mathrm{sp}}_{W_2(k),\log}$ with closed fiber $(X,D)$,  and set 
$$R\Gamma(\tilde X,\tilde D):=(R\Gamma(X,(\Omega^*_{X/k}(\log D),W,\Fil)),R\Gamma(X',\Psi_{(\tilde X,\tilde D)}));$$
\item[-\ ] for any morphism $\tilde f:(\tilde X,\tilde D)\to(\tilde Y,\tilde E)$ in $\mathrm{Sch}^{\mathrm{sp}}_{W_2(k),\log}$ with closed fiber $f:(X,D)\to(Y,E)$,  we set 
$$R\Gamma(\tilde f):=(f^*:R\Gamma(Y,\Omega^*_{Y/k}(\log E))\to R\Gamma(X,\Omega^*_{X/k}(\log D))).$$
\end{itemize}
\end{definition}
The following easily checked lemma finishes the proof of Theorem \ref{main theorem}.
\begin{lemma}
The functors $\rho^{\mathrm{MFLC}}_\dR\circ R\Gamma,\rho^{\mathrm{MFLC}}_\Hig\circ R\Gamma$ coincide with the functors $\rho_\dR,\rho_\Hig$ in Theorem \ref{main theorem}. Together with the Hodge filtration $\Fil:=\{\Fil^l\rho^{\mathrm{MFLC}}_\dR\circ R\Gamma\}_l$ on $\rho^{\mathrm{MFLC}}_\dR\circ R\Gamma$ and an isomorphism of functors
$$\psi\circ R\Gamma:\rho^{\mathrm{MFLC}}_\Hig\circ R\Gamma\to \rho^{\mathrm{MFLC}}_\dR\circ R\Gamma,$$
Theorem \ref{main theorem} follows.
\end{lemma}

\end{document}